\RequirePackage{fix-cm}
\documentclass[11pt,a4paper]{article}
\usepackage{graphicx}
\usepackage[margin=1.4in]{geometry}
\usepackage{amsmath,amssymb,amsthm}
\usepackage{bm}
\usepackage{tikz}
\usetikzlibrary{math}
\usetikzlibrary{intersections}
\usetikzlibrary{backgrounds}
\usetikzlibrary{calc}
\usepackage{enumitem}
\setlist{noitemsep}
\SetEnumitemKey{inc:def:counter}{ref=\thedefinition\alph*,topsep=0pt,label=\alph*)}
\usepackage[colorlinks]{hyperref}
\hypersetup{
    linkcolor=red!50!black,
    citecolor=blue
}
\newcommand{\IR}{\mathbb{IR}}
\newcommand{\R}{\mathbb{R}}

\renewcommand{\b}[1]{\bm{#1}}
\newcommand{\defeq}{\mathrel{\mathop:}=} %aligned operator
\newcommand{\T}{{}^{\mathrm{T}}}
\newcommand{\LP}{\textsl{LP}}
\newcommand{\ILP}{\textsl{ILP}}
\newcommand{\LS}{\textsl{LS}}
\newcommand{\ILS}{\textsl{ILS}}

\DeclareMathOperator{\diag}{diag}
\tikzmath{
        function comparecoords(\ax,\ay,\bx,\by) {
            int \ret;
            \ret = 0;
            if and(\ax>=\bx,\ay>=\by) then {
                \ret = 1;
            } 
            else { if and(\ax<=\bx,\ay<=\by) then {
                \ret = -1;
            } else
        {\ret = 0;}; };
            return \ret;
        };
}

\newcommand{\seg}[4][none]{
    \def\tempi{none}
    \def\temp{#1}
    \ifx\tempi\temp\def\tempswitch{0}\else\def\tempnamepath{\temp}\def\tempswitch{1}\fi
    \tikzmath{
        real \aaa, \bbb, \ccc, \ddd, \eee, \fff ;
        coordinate \sowe,\noea;
        \sowe = (current bounding box.south west);
        \noea = (current bounding box.north east);
        \aaa = #2;
        \bbb = #3;
        \ccc = #4 cm;
        int \isse;
        \isse = 0;
        coordinate \issecoo;
        coordinate \li, \ri, \ti, \bi;
        coordinate \starta, \enda;
        \starta = (\sowe) + (-1,-1);
        \enda = (\sowe)+(-1,-1);
        if \aaa == 0 then {
            \starta = (\sowex, \ccc/\bbb);
            \enda = (\noeax, \ccc/\bbb);
        }
        else {
            if \bbb == 0 then {
                \starta = (\ccc/\aaa, \sowey);
                \enda = (\ccc/\aaa, \noeay);
            }  
            else {
                \ddd = 1 pt;
                \eee = 0 pt;
                \fff = \sowex;
                \retx = (\fff*\bbb-\eee*\ccc) / (\ddd*\bbb-\eee*\aaa);
                \rety = (\fff*\aaa-\ddd*\ccc) / (\eee*\aaa-\ddd*\bbb);
                \li = ( \retx pt,\rety pt );
                \ddd = 1 pt;
                \eee = 0 pt;
                \fff = \noeax;
                \rety = (\fff*\aaa-\ddd*\ccc) / (\eee*\aaa-\ddd*\bbb);
                \retx = (\fff*\bbb-\eee*\ccc) / (\ddd*\bbb-\eee*\aaa);
                \ri = ( \retx pt ,\rety pt );
                \ddd = 0 pt;
                \eee = 1 pt;
                \fff = \noeay;
                \retx = (\fff*\bbb-\eee*\ccc) / (\ddd*\bbb-\eee*\aaa);
                \rety = (\fff*\aaa-\ddd*\ccc) / (\eee*\aaa-\ddd*\bbb);
                \ti = ( \retx pt,\rety pt );
                \ddd = 0 pt;
                \eee = 1 pt;
                \fff = \sowey;
                \rety = (\fff*\aaa-\ddd*\ccc) / (\eee*\aaa-\ddd*\bbb);
                \retx = (\fff*\bbb-\eee*\ccc) / (\ddd*\bbb-\eee*\aaa);
                \bi = ( \retx pt,\rety pt );
                if and(comparecoords(\li,\sowe),comparecoords(\noea,\li)) then {
                    \isse = 1;
                    \starta = (\li);
                };
                if and(comparecoords(\ti,\sowe),comparecoords(\noea,\ti)) then {
                    if \isse == 0 then {
                        \isse = 1;
                        \starta = (\ti);
                    }
                    else {
                        \enda = (\ti);
                    };
                };
                if and(comparecoords(\ri,\sowe),comparecoords(\noea,\ri)) then {
                    if \isse == 0 then {
                        \isse = 1;
                        \starta = (\ri);
                    }
                    else {
                        \enda = (\ri);
                    };
                };
                if and(comparecoords(\bi,\sowe),comparecoords(\noea,\bi)) then {
                    if \isse == 0 then {
                        \isse = 1;
                        \starta = (\bi);
                    }
                    else {
                        \enda = (\bi);
                    };
                };
            };
        };
        if \tempswitch == 0 then {
            { \draw (\starta) -- (\enda); };
        }
        else {
            { \coordinate (\tempnamepath-st) at (\starta); };
            { \coordinate (\tempnamepath-en) at (\enda); };
            { \draw [name path global=\tempnamepath] (\starta) -- (\enda); };
        };
    }
}

\newcommand{\defineintervalsymbols}[2][]{%
    \def\temp{#1}
    \ifx\temp\empty
    \def\comname{#2}
    \else
    \def\comname{#1}
    \expandafter\newcommand\csname #1\endcsname{\ensuremath{#2}}
    \fi
    \expandafter\newcommand\csname d\comname\endcsname{\ensuremath{\underline{#2}}}
    \expandafter\newcommand\csname u\comname\endcsname{\ensuremath{\overline{#2}}}
    \expandafter\newcommand\csname b\comname\endcsname{\ensuremath{\bm{#2}}}
    \expandafter\newcommand\csname m\comname\endcsname{\ensuremath{#2{}^c}}
    \expandafter\newcommand\csname r\comname\endcsname{\ensuremath{#2{}^\Delta}}
}

\defineintervalsymbols{A}
\defineintervalsymbols{B}
\defineintervalsymbols{a}
\defineintervalsymbols{b}
\defineintervalsymbols[be]{\beta}
\defineintervalsymbols[Af]{A^{\mathrm{f}}}
\defineintervalsymbols[An]{A^{\mathrm{n}}}
\defineintervalsymbols[Bf]{B^{\mathrm{f}}}
\defineintervalsymbols[Bn]{B^{\mathrm{n}}}
\defineintervalsymbols[cf]{c^{\mathrm{f}}}
\defineintervalsymbols[cn]{c^{\mathrm{n}}}
\defineintervalsymbols[yf]{y^{\mathrm{f}}}
\defineintervalsymbols[yn]{y^{\mathrm{n}}}
\defineintervalsymbols[xf]{x^{\mathrm{f}}}
\defineintervalsymbols[xn]{x^{\mathrm{n}}}
\defineintervalsymbols[xp]{x^+}
\defineintervalsymbols[ap]{a^+}
\defineintervalsymbols[bp]{b^+}
\defineintervalsymbols[Afp]{A^{\mathrm{f}+}}
\defineintervalsymbols[Anp]{A^{\mathrm{n}+}}
\defineintervalsymbols[Bfp]{B^{\mathrm{f}+}}
\defineintervalsymbols[Bnp]{B^{\mathrm{n}+}}
\defineintervalsymbols[cfp]{c^{\mathrm{f}+}}
    \defineintervalsymbols[cnp]{c^{\mathrm{n}+}}
    \defineintervalsymbols[yfp]{y^{\mathrm{f}+}}
    \defineintervalsymbols[ynp]{y^{\mathrm{n}+}}
    \defineintervalsymbols[xnp]{x^{\mathrm{n}+}}
    \defineintervalsymbols[xfp]{x^{\mathrm{f}+}}

\newcommand{\scs}{s^\mathrm{s}}
\newcommand{\scp}{s^\mathrm{p}}
\newcommand{\Ds}{D^\mathrm{s}}
\newcommand{\Dp}{D^\mathrm{p}}
\newtheorem{theorem}{Theorem}

\newtheorem{lemma}{Lemma}
\newtheorem{corollary}{Corollary}

\theoremstyle{definition}

\newtheorem{definition}{Definition} 
\newtheorem{example}{Example}
\newtheorem{remark}{Remark}

\begin{document}

\title{Testing weak optimality of a given solution in interval linear programming revisited: NP-hardness proof, algorithm and some polynomial cases\thanks{The work of M.~Rada was supported by the Czech Science Foundation Grant P403/17-13086S. The work of M.~Hlad\'ik and E. Garajov\'a was supported by the Czech Science Foundation Grant P402/13-10660S. The work of E. Garajov\'a was also supported by Grant no. 156317 of Grant Agency of Charles University and by the grant SVV-2017-260452.}
}

%\titlerunning{Testing weak optimality in interval LP: NP-hardness proof, algorithm, some polynomial cases}        % if too long for running head

\author{Miroslav Rada\footnote{
University of Economics, Faculty of Finance and Accounting, Department of Financial Accounting and Auditing, W. Churchill's Sq. 4, 130 67 Prague, Czech Republic,
e-mail: \texttt{miroslav.rada@vse.cz}
}\and
        Milan Hlad\'ik\footnote{
Charles University, Faculty  of  Mathematics  and  Physics,
Department of Applied Mathematics, 
Malostransk\'e n\'am.~25, 11800, Prague, Czech Republic, 
e-mail: \texttt{milan.hladik@matfyz.cz}
}\and
        Elif Garajov\'a\footnote{
Charles University, Faculty  of  Mathematics  and  Physics,
Department of Applied Mathematics, 
Malostransk\'e n\'am.~25, 11800, Prague, Czech Republic, 
e-mail: \texttt{elif@kam.mff.cuni.cz}
}
}

\maketitle

\begin{abstract}
    We address the problem of testing weak optimality of a given solution of a given interval linear program. The problem was recently wrongly stated to be polynomially solvable. We disprove it. We show that the problem is NP-hard in general. We propose a new algorithm for the problem, based on orthant decomposition and solving linear systems. Running time of the algorithm is exponential in the number of equality constraints. Interval linear programs with inequality constraints only can be processed in polynomial time.
%\subclass{65G40 \and 90C31}
\end{abstract}


\section{Introduction}
\label{sec:introduction}
In this paper, we address the following problem:

``\emph{Given an interval linear program, decide whether a given weakly feasible solution is weakly optimal}''.

This problem was recently wrongly stated to be polynomially solvable in \cite{li:2015:Checkingweakoptimality}. Our aims are the following:
\begin{itemize}
    \item to show a counterexample to the method proposed by \cite{li:2015:Checkingweakoptimality} and explain what is wrong in their proof,
    \item to show that the problem is actually NP-hard in general,
    \item to propose an algorithm for the problem,
    \item to describe some polynomial cases.
\end{itemize}

\paragraph{Structure of the paper.} In Section \ref{sec:preliminaries}, we introduce the notion of interval linear programming and define our problem formally. %, allowing us to define our problem formally in Section \ref{sec:our:problem}. 
Section \ref{sec:motivation} provides an overview of related work and some motivation for our paper. The introductory part of the paper is finalized by Section \ref{sec:counterexamples}, which provides a counterexample to the method proposed in~\cite{li:2015:Checkingweakoptimality} and also points to the weakness in the proof therein.% which claimed faulty that the problem was solved and that a polynomial algorithm exists for weak optimality testing in its full generality.

The main results of the paper are contained in Sections \ref{sec:np:hardness} and \ref{sec:polynomial:case}. In the former section, we prove that our problem is NP-hard (via reduction from testing solvability of interval linear systems). In the latter section, we prove that weak optimality of~a~given solution of~a~given interval linear program can be tested by solving $2^k$ linear programs (of the same size), where $k$ is the number of equality constraints in the interval linear program.

In particular, this means that if an interval linear program contains only inequality constraints, weak optimality of a given solution can be tested in polynomial time with one linear program. More generally, the test can be performed in polynomial time as long as the number of equality constraints remains ``small'' (for example constant or logarithmic).

To avoid confusion at this point, we premise that reasons of distinguishing equality and inequality constraints (which is not necessary in classical linear programming) will be clarified in Remark \ref{rem:distinguishing:constraints:variables}.

\subsection{Notation, intervals and interval linear programming}\label{sec:preliminaries}

For two real matrices $\dA,\uA \in \R^{m\times n}$ such that $\dA \le \uA$, an \emph{interval matrix} is the set of matrices $\b{A} \defeq \{A\in \R^{m\times n}:\ \dA \le A \le \uA\}$. The set of all interval matrices of dimension $m\times n$ is denoted by $\IR^{m\times n}$. Interval vectors and scalars are defined analogously. 

The multiplication of a real and an interval is defined as follows. Assume $\alpha \in \R$ and $[\da, \ua] \in \IR$. If $\alpha < 0$, then $\alpha[\da,\ua] = [\alpha\ua,\alpha\da]$, otherwise $\alpha[\da,\ua] = [\alpha\da,\alpha\ua]$.

Throughout the paper, bold symbols are reserved for interval matrices, vectors and scalars, while symbols in italics represent real structures. 
The symbol $0$ denotes the zero matrix or vector of suitable dimension. Also, $e$ is the vector of suitable dimension containing ones. Generally, we omit declaration of dimensions of matrix or vector variables wherever no confusion should arise. Vectors are understood columnwise.% with one exception: objective vectors will be given as row vectors.

The $i$th row of a (possibly interval) matrix $A$ is denoted by $A_i$. For vectors, the lower index points to the concrete element.
%The $j$-th column of a (possibly interval) matrix $\spadesuit$ is denoted by $\spadesuit_{:,j}$. 
The symbol $\diag(a)$ for $a\in \R^n$ is the diagonal matrix with entries of $a$.

Definition \ref{def:interval:linear:programming:problem} introduces the notion of interval linear programming and also the interval linear systems.  

\begin{definition}[Interval linear programming]
  \label{def:interval:linear:programming:problem}\hfill
  \begin{enumerate}[inc:def:counter]
    \item Let the following interval matrices and vectors with dimensions in brackets be given: $$\bAf ({k \times m}),\bAn ({k \times n}),\bBf ({\ell \times m}),\bBn (\ell \times n),$$ accompanied with interval vectors $\ba \in \IR^k, \bb \in \IR^{\ell}, \bcf \in \IR^m, \bcn \in \IR^n$. 

        Define $\Ds \defeq \bAf \times \bAn \times \bBf \times \bBn \times \ba \times \bb $ and $\Dp \defeq \Ds \times \bcf \times \bcn$. The sets $\Dp$ will be called \emph{data of an interval linear program}. Analogously, $\Ds$ will be \emph{data of an interval linear system}.
    \item Denote a tuple from $\Ds$ by  $$\scs \defeq (\Af,\An,\Bf, \Bn, a, b).$$
      Any tuple from $\Dp$, say
      $$ \scp \defeq (\scs, \cf, \cn) \in \Dp$$ is called \emph{scenario} of interval linear program. Sometimes, also $\scs$ will be called \emph{scenario} of interval linear system.

      To every scenario, a linear program~\eqref{eq:lp} (shortly LP), denoted by $\LP(\scp)$, is associated:
      \begin{subequations}\label{eq:lp}%
  \begin{alignat}{3}%
      \min_{\xf, \xn}\ &&\phantom{=}%_{\xf \in \R^{\nf},\xn \in \R^{\nn}}
    \label{eq:lp:objective}(\cf)\T  \xf &+& (\cn)\T \xn & \quad \text{s.t.}\\
    \label{eq:lp:equalities} &&\phantom{=}\Af \xf &+& \An \xn & = a,\\
    \label{eq:lp:inequalities}        &&\phantom{=}\Bf \xf &+& \Bn \xn & \ge b,\\
    \label{eq:lp:nonnegativity}        && &&\phantom{=}\xn & \ge 0.
\end{alignat}
\end{subequations}

Analogously, to every $\scs \in \Ds$, a linear system \eqref{eq:lp:equalities}--\eqref{eq:lp:nonnegativity} is assigned; such a~system is denoted by $\LS(\scs)$.
    \item An \emph{interval linear program} (shortly ILP) with data $\Dp$, denoted by $\ILP(\Dp)$, is the family of linear programs $\{\LP(\scp):\,\scp \in \Dp\}$. 
    \item\label{def:interval:linear:system} An \emph{interval linear system} with data $\Ds$, denoted by $\ILS(\Ds)$ is the family of linear systems $\{\LS(\scs):\, \scs \in \Ds\}$.
  \end{enumerate}
\end{definition}
To simplify, an interval linear program is the family of linear programs with coefficients varying along given intervals. Similarly an interval linear system is the family of linear systems.

We will use interval linear programs and systems quite often. For convenience and readability, we will write them in short form: the interval linear program with data~$\Dp$ reads%
\begin{subequations}\label{eq:ilp}%
  \begin{alignat}{3}%
    \min\ &&\phantom{=}%_{\xf \in \R^{\nf},\xn \in \R^{\nn}}
    (\bcf)\T  \xf &+& (\bcn)\T \xn &\quad \text{s.t.}   \\
&&\phantom{=}\bAf \xf &+& \bAn \xn & = \ba,\\
            &&\phantom{=}\bBf \xf &+& \bBn \xn & \ge \bb,\\
            && &&\phantom{=}\xn & \ge 0,
\end{alignat}
\end{subequations}
the interval linear system will be written in  an analogous way.

\begin{remark}[on notation]
Some symbols in the paper have upper indices. These should simplify orientation in the (not very small) amount of different symbols. The indices are ``$\mathrm{p}$'' for ``of a program'', ``$\mathrm{s}$'' for ``of a system'', ``$\mathrm{f}$'' for ``free'' (variables) and ``$\mathrm{n}$'' for ``nonnegative'' (variables).
\end{remark}

\begin{remark}\label{rem:distinguishing:constraints:variables}
    Unlike for linear programming, the distinguishing of inequality and equality constraints  does matter here. The mutual transformation of the types of constraints is not possible in general. For example, an equality constraint $\b{A}x = b$ cannot be rewritten to the system $\b{A}x \le b, \b{A}x \ge b$. The former means $Ax \le b, Ax \ge b$ for all $A \in \b{A}$, while the latter reads $A_1x \le b, A_2 x \ge b$ for all $A_1\in \b{A}, A_2 \in \b{A}$. The problem is that we lose the dependency $A_1 = A_2$ during the transformation -- one calls this the \emph{dependency problem}. For further details, examples and possible transformations, see~\cite{GarHla2017a,Hla2017a}.

    The distinguishing of nonnegative and free variables has a bit different primary background. In fact, it turns out that many questions in interval analysis are much simpler with nonnegative variables than with free variables. Hence, the types of variables are often treated separately. For our results, this distinction is not really necessary. We do so just to demonstrate the generality of our results.
\end{remark}

\paragraph{Feasibility and optimality in classical linear programming.} Properties of \emph{feasibility} and \emph{optimality} are well known when dealing with linear systems or programs. This holds also for \emph{feasible} or \emph{optimal solution} of a linear system or program.  

For all the above properties, one can define a decision problem in form ``does the particular property hold for a given (solution of) linear program/system?''. Note that all such decision problems can be solved using algorithms for linear programming.

%\begin{itemize}
    %\item Consider a linear system $\LS(\scs)$ for a given $\scs$. This system is called \emph{feasible}, if there exists $x = \binom{\xf}{\xn}$ such that it satisfies \eqref{eq:lp:equalities}--\eqref{eq:lp:nonnegativity} for $\scs$. Also, a given $x$ is \emph{feasible solution of $\LS(\scs)$}, if it satisfies \eqref{eq:lp:equalities}--\eqref{eq:lp:nonnegativity} for $\scs$.
    %\item Consider a linear program $\LP(\scp)$ for a given $\scp$. 
        %\begin{itemize}
            %\item The program $\LP(\scp)$ is called \emph{feasible} if the underlying linear system is feasible, similarly, a given solution is \emph{a feasible solution} if it is a feasible solution of the underlying linear system.
            %\item The program $\LP(\scp)$ is \emph{optimal} (or \emph{has an optimum}), if there exists $x$ such that the inequality $cx' \ge cx$ holds for all feasible $x'$. Also, a given $x$ is \emph{optimal solution of $\LP(\scp)$}, if the inequality $cx' \ge cx $ holds for all feasible $x'$. 
%{\mh[Nevim jestli pro tento journal nutne pripominat co je optimum. Mozna bych jen vypich, ze optimalni = pripustne a ne neomezene.]}
        %\end{itemize}
%\end{itemize}

%For all the above properties, one can define a decision problem in form ``does the particular property hold for a given linear program/system?''. Note that all such decision problems can be solved using algorithms for linear programming.

\paragraph{Feasibility and optimality in interval linear programming.}
For interval linear systems and programs, the above properties and associated decision problems are not so straightforward to formulate.
There are at least two quite natural ways to build analogous problems. For our paper, the analogies that could be called \emph{weak} problems are interesting. In Definitions \ref{def:weak:feasibility} and \ref{def:weak:optimality}, we build analogies to all the above properties. Actually, we are especially interested in Definitions \ref{def:weak:feasibility:problem} and \ref{def:weak:optimal:solution}.
For other concepts of feasibility and optimality in interval linear programming, see Remark \ref{rem:other:concepts:optimality}.

%\paragraph{Feasibiity and optimality in interval linear programming.}
%In a classical linear program, say $\LP(\scp)$, a feasible solution $x^* = (\xf{}^*, \xn{}^*)\T$ \emph{is optimal} if the inequality $\cf \xf{}^* + \cn \xn{}^* \ge \cf\xf + \cn \xn$ holds for all feasible $x = (\xf, \xn)\T$.

\begin{definition}[Weak feasibility]
  \label{def:weak:feasibility}
  Assume that data $\Ds$ of a system $\ILS(\Ds)$ are given. 
  \begin{enumerate}[inc:def:counter]
      \item\label{def:weak:feasibility:problem} The system $\ILS(\Ds)$ is \emph{weakly feasible}, if there exists $\scs \in \Ds$ such that $\LS(\scs)$ is feasible. 
      \item\label{def:weak:feasible:solution} A given $x = (\xf,\xn)$ is a \emph{weakly feasible solution} of $\ILS(\Ds)$, if there exists $\scs \in \Ds$ such that $x$ is a feasible solution of $\LS(\scs)$.
  \end{enumerate}
\end{definition}

\begin{definition}[Weak optimality]
    \label{def:weak:optimality} Assume that data $\Dp$ of a program $\ILP(\Dp)$ are given.
  \begin{enumerate}[inc:def:counter]
      \item \label{def:weak:optimality:problem}
          The program $\ILP(\Dp)$ is \emph{weakly optimal}, if there exists $\scp \in \Dp$ such that $\LP(\scp)$ has an optimum.
      \item \label{def:weak:optimal:solution}
          A given $x = (\xf,\xn)$ is a \emph{weakly optimal solution} of $\ILP(\Dp)$, if there exists $\scp \in \Dp$ such that $x$ is an optimal solution of $\LP(\scp)$.
  \end{enumerate}
\end{definition}

The formulation of the problem we are facing follows, in two variants.
%\spnewtheorem*{oup}{Our problems}{\bf}{}
\paragraph*{Our problems}
%\begin{oup}\hfill
\begin{enumerate}[label=(P\arabic*),leftmargin=*,topsep=0pt]
    \item\label{enu:wopt}
        Given data $\Dp$ of an interval linear program and a solution $x = (\xf, \xn)$, test whether $x$ is weakly optimal, i.e. decide whether $x$ is optimal for some scenario $\scp \in \Dp$.
    \item\label{enu:cwopt} Decide \ref{enu:wopt}. If the answer is \emph{yes}, find a scenario witnessing it.
\end{enumerate}

The problem \ref{enu:cwopt} is a constructive version of \ref{enu:wopt}. Note also that a scenario is a~sufficient witness of weak optimality.
%\end{oup}

\begin{remark}
    \label{rem:other:concepts:optimality}
    Note that the weak problems ask in general the following question: ``Given a property of classical linear program/system, is the property satisfied for at least one scenario of a given interval linear program/system?''.
    If the quantifier ``at least one scenario'' is interchanged for ``every scenario'', one obtains \emph{strong} problems. For example, a feasible $x$ is a strongly optimal solution of a given $\ILP(\Dp)$, if $x$ is an optimal solution of $\LP(\scp)$ for every $\scp \in \Dp$.
The survey on results regarding both the weak and strong problems in interval linear programming can be found in \cite{hladik:2012:Intervallinearprogramming} and the corresponding problems related to interval linear systems of equations and inequalities in \cite{fiedler:2006:Linearoptimizationproblems}.
\end{remark}

\subsection{Related work}\label{sec:motivation}
This problem naturally emerged when dealing with various questions and problems regarding interval linear programming. A survey on results can be found in \cite{hladik:2012:Intervallinearprogramming}. Since that time, a partial characterization of the weakly optimal solution set was given in \cite{AshNeh2018} and an inner approximation was considered in \cite{AllNeh2013}.
More general concepts of solutions, extending weak and strong solutions, were recently addressed in \cite{Li2015,LiLiu2015,LuoLi2014a}. Particular quantified solutions were also studied in \cite{Hla2016c,Hla2017d}. Duality in interval linear programming, which helps in charaterizing of weak optimality, among others, was studied in \cite{NovHla2017a}.

\subsection{Counterexample and the weakness in the former proof}\label{sec:counterexamples}
We provide an example on which the general polynomial method for checking weak optimality presented in \cite{li:2015:Checkingweakoptimality} fails.

\begin{example}
    Consider the $\ILP((0,([0,2],[0,2]),0,0,2,0,0,(0,1)\T))$ with two nonnegative variables and one equality constraint
    \begin{equation*}
  \begin{alignedat}{2}%
    \min\ &&\phantom{=}%_{\xf \in \R^{\nf},\xn \in \R^{\nn}}
    (0,1) \xn &\quad \text{s.t.}   \\
              &&\phantom{=}([0,2],[0,2]) \xn & = 2,\\
              &&\phantom{=} \xn & \ge 0.\\
  \end{alignedat}
  \end{equation*}
  Assume we want to test the weak optimality of $x = \binom{1}{1}$. The method of \cite{li:2015:Checkingweakoptimality} selects the scenario with $\An = (1,1)$ and says that $x$ is weakly optimal if and only if the system with variables $y^1 \in \R, c^1 \in \R^2$ (we use the notation of the original paper)
  \begin{align*}
      y^1 \textstyle\binom{1}{1}  = c^1,&\qquad \text{ \small (dual feasibility constraint, (6a) in the original paper)}\\
      \textstyle\binom{0}{1} \le c^1 \le \binom{0}{1},& \qquad \text{ \small (scenario feasibility constraint, (6f) in the original paper)}
  \end{align*}
if feasible.

This yields the result that $x$ is not weakly optimal. This is wrong, since $x$ is optimal for a scenario with $\An = (0,2)$.

%The problem is that the scenario was chosen in advance.
%Unfortunately, this cannot be fixed by means of any polynomial method, since we will show in Theorem \ref{the:weak:optimality:np:hard} that it is an NP-hard problem in general.
\end{example}

The problem is that the method selects a fixed scenario (see the system (4) in the original paper) without taking \emph{dual program} into account. In particular, the objective vectors do not influence the selected scenario at all.
Then, the weakness of the proof of the main theorem (Theorem 3.1 in the original paper) is in the ``only if'' part, namely in the first sentence on page 84. The paper states that there exists a solution satisfying the systems (20) and (21) therein (read ``there exists an optimal solution of the dual program''), however, this is not ensured, since the scenario was chosen to satisfy the primal feasibility only.

Unfortunately, this cannot be fixed by means of any polynomial method, since we will show in Section \ref{sec:np:hardness} that it is an NP-hard problem in general.

\section{Auxiliary result -- strong duality in LP}\label{sec:auxiliary:results}
In the next sections, we will strongly rely on the obvious characterization of the set of optimal solutions of an interval linear program using strong duality theorem for linear programming.
\begin{lemma}[Characterization of weak optimality using strong duality]
    \label{lem:characterization:weak:optimality}
    Consider an interval linear program $\ILP(\Dp) = \ILP(\bAf, \bAn, \bBf, \bBn, \ba, \bb, \bcf, \bcn)$. The solution $x = (\xf, \xn)$ is a weakly optimal solution of $\ILP(\Dp)$, if and only if $(\xf, \xn, \yf, \yn, \scp)$ is a feasible solution of the system 
    \begin{subequations}
        \label{eq:strong:duality}
  \begin{alignat}{3}%
    &&\Af \xf &+& \An \xn & = a,\nonumber\\
      &&\Bf \xf &+& \Bn \xn & \ge b,\label{eq:strong:duality:primal}\\
     && &&\xn & \ge 0,\nonumber\\
     &&(\Af)\T \yf &+& (\Bf)\T \yn &= \cf,\nonumber\\
      &&(\An)\T \yf &+&  (\Bn)\T \yn &\le \cn,\label{eq:strong:duality:dual}\\
    && && \yn &\ge 0,\nonumber\\
      %&& \cf \xf &+ & \cn \xn & = a\T \yf + b\T \yn,\label{eq:strong:duality:complementarity}
      && \yn_i(b - \Bf\xf &-& \Bn \xn)_i &= 0, \quad i =1,\dots,\ell,\label{eq:strong:duality:complementarity}\\
      && \xn_i(\cn - (\An)\T\yf &-& (\Bn)\T \yn)_i &= 0, \quad i = 1, \dots n,\label{eq:strong:duality:complementarity:primal}
\end{alignat}%
    \begin{equation}
        \label{eq:characterization:weak:optimality:constraints}
    \begin{alignedat}{4}
        \Af & \in \bAf, &\quad \An &\in \bAn, &\quad \Bf & \in \bBf, &\quad \Bn &\in \bBn, \\
        a & \in \ba,    &\quad  b & \in \bb, &\quad \cf &\in \bcf, &\quad \cn &\in \bcn
    \end{alignedat}
\end{equation}
\end{subequations}
    for some $\yf , \yn$ and $(\Af, \An, \Bf, \Bn, a, b, \cf, \cn) = \scp$.
\end{lemma}

For a fixed scenario $\scp$, the constraints \eqref{eq:strong:duality:primal} correspond to the feasibility of primal program, the constraints \eqref{eq:strong:duality:dual} to the feasibility of dual program, and the constraints \eqref{eq:strong:duality:complementarity} and \eqref{eq:strong:duality:complementarity:primal} to the complementary slackness.
\begin{proof}[of Lemma \ref{lem:characterization:weak:optimality}]
    If $x=(\xf,\xn)$ is a weakly optimal solution, there exists a~scenario $\scp \in \Dp$ such that $x$ is an optimal solution of $\LP(\scp)$. Using the well known strong duality theorem we know that there exists a~tuple $(\xf, \xn, \yf, \yn)$ such that $(\xf, \xn, \yf, \yn, \scp)$ solves \eqref{eq:strong:duality}. 

    Similarly, if $(\xf, \xn, \yf, \yn, \scp)$ is a feasible solution of \eqref{eq:strong:duality}, we obtain that $(\xf, \xn)$ is a weakly optimal solution of $\ILP(\Dp)$ for scenario $\scp$ from application of the strong duality theorem.
\end{proof}

\begin{remark}
Note that the system \eqref{eq:strong:duality}
\begin{itemize}[topsep=0pt]
    \item is nonlinear and remains nonlinear even for fixed $(\xf, \xn)$,
    %\item is linear for fixed $(\xf, \xn, \yn)$, {\mh[Nestaci jen $(\xn, \yn)$?]}\red{vyhodit?}
    \item can be rewritten to a linear system for fixed $\scp$,
    \item is \emph{not} a standard interval linear system due to the dependency problem: note the multiple occurrences of individual coefficients that could be considered interval coefficients (see Remark \ref{rem:distinguishing:constraints:variables}). It is rather a linear parametric system, where parameters attain values from given intervals. This is actually what makes the weak optimality harder to grasp than weak feasibility for interval linear program (cf.~this with classical linear program, where optimality and feasibility are essentially the same problems).
\end{itemize}
\end{remark}

\begin{remark}[Geometry of the strong duality -- implications on ILPs]
    A solution $x$ of a linear program is optimal if it is feasible (primal feasibility constrains \eqref{eq:strong:duality:primal}) and if the objective vector can be obtained as combination of normals of \emph{active} constraints (this is the dual feasibility \eqref{eq:strong:duality:dual}), i.e. such constraints that are satisfied as equalities (the first complementarity condition \eqref{eq:strong:duality:complementarity} controls that only active constraints can be used in combination). Coefficient $\cn_i$ of a nonnegative variable may be higher than the summed coefficients of scaled active normal, but only if corresponding variable $\xn_i$ is zero (the second complementarity condition \eqref{eq:strong:duality:complementarity:primal}).
Note also that equality constraints are active ``in both directions'' -- their normals can be used also with negative coefficients in the combination.

Geometrically, dual variables $\yf$ and $\yn$ scale the normals of active primal constraints and sum them. If there is a way how to scale and sum them up to the objective vector, it means that the primal solution is optimal.

In classical linear programming, the optimality test of $x$ is easy since all the normals of active constraints are known. In interval linear programming, the task is not so straightforward, since for different scenarios we can have different normals and possibly also different sets of active constraints. However, the set of all possible normals for each individual constraint can definitely  be expressed as a polytope. In fact, this is exactly what the inequalities \eqref{eq:strong:duality:primal} and \eqref{eq:characterization:weak:optimality:constraints} do (for a fixed $x$). However, the conic hull constraints \eqref{eq:strong:duality:dual} are nonlinear, since coefficients of conic combination are multiplied by the active normals. 
\end{remark}
%\begin{proposition}[Characterization of optimality using strong duality]
    %\label{pro:strong:duality}
    %Consider a linear program $\LP(\scp) = \LP(\Af,\An,\Bf,\Bn,a,b,\cf,\cn)$. A solution $x = (\xf, \xn)$ of $\LP(\scp)$ is optimal if and only if $(\xf,\xn, \yf, \yn)\T$ is the feasible solution of the system %

      %\begin{subequations}\label{eq:strong:duality}%
  %\begin{alignat}{3}%
    %\label{eq:strong:duality:primal:equalities} &&\Af \xf &+& \An \xn & = a,\\
    %\label{eq:strong:duality:primal:inequalities}        &&\Bf \xf &+& \Bn \xn & \le b,\\
    %\label{eq:strong:duality:primal:nonnegativity}        && &&\xn & \ge 0,\\
    %\label{eq:strong:duality:dual:equalities} && \Af\T \yf &+& \Bf\T \yn &= \cf,\\
    %\label{eq:strong:duality:dual:inequalities}&& \An\T \yf &+& \Bn\T \yn &\le \cn,\\
    %\label{eq:strong:duality:dual:nonnegativity}&& && \yn &\ge 0,\\
    %\label{eq:strong:duality:complementarity} && \cf \xf &+ & \cn \xn & = a \yf + b \yn
%\end{alignat}
%\end{subequations}
%for some $\yf, \yn$.
%\end{proposition}

%Characterization of the set of weak optimal solutions of interval linear program can be easily derived from Proposition \ref{pro:strong:duality}.

\section{NP-hardness proof}\label{sec:np:hardness}
In this section, we prove that the problem \ref{enu:wopt} is NP-hard. We show this by a reduction from weak feasibility testing of an interval system of inequalities with free variables. We lean on the well-known result stated in Lemma \ref{the:weak:feasibility:np:hard}.
\begin{lemma}[NP-hardness of weak feasibility, Rohn {\cite[p. 58]{fiedler:2006:Linearoptimizationproblems}}]
    \label{the:weak:feasibility:np:hard}
    Consider the family of interval linear systems with free variables and inequality constraints only, in form 
    \begin{equation}
        \bBf \xf \le \bb,
        \label{eq:inequality:free:variables}
    \end{equation}
    i.e. the family of interval systems with data $\Ds \defeq (0, 0, -\bBf, 0, 0, -\bb)$. 

    The problem ``given data $\Ds$, decide whether $\ILS(\Ds)$ is weakly feasible'' is NP-hard.
\end{lemma}

Our result follows:
\begin{theorem}
    \label{the:weak:optimality:np:hard}
    The problem \ref{enu:wopt} is NP-hard.    
\end{theorem}
\begin{proof}
    Consider the family of interval linear programs with nonnegative variables and equality constraints only, i.e. the family of interval linear programs with data in form $\Dp = (0, \bAn, 0, 0, 0, 0, 0 ,\bcn)$. 

    Using Lemma \ref{lem:characterization:weak:optimality} we know that a given $x$ is a weakly optimal solution of $\ILP(\Dp)$ if and only if the system \eqref{eq:strong:duality} has a solution for the fixed $x$. Hence, the system reads (zero rows and summands are omitted)
    \begin{equation}
        \label{eq:wopt:np:hard:short}
    \begin{alignedat}{2}
        \An \xn &= 0, \\
            \xn &\ge 0, \\
            (\An)\T \yf & \le \cn, \\
            \xn_i ( \cn - (\An)\T \yf)_i        & = 0,& \qquad i=1,\dots,n,\\
            %\cn \xn &= 0,\\
            \An \in \bAn, \quad \cn &\in \bcn.&
\end{alignedat}
\end{equation}
Now, assume that we want to test weak optimality of the solution $\xn = 0$. The system \eqref{eq:wopt:np:hard:short} becomes 
\begin{subequations}
    \begin{alignat}{1}
        (\An)\T \yf & \le \cn , \label{eq:wopt:np:hard:shortest:inequalities}\\ 
        \An \in \bAn, \quad \cn &\in \bcn.
\end{alignat}
    \label{eq:wopt:np:hard:shortest}
\end{subequations}

We have that $\xn =  0$ is a weakly optimal solution of $\ILP(\Dp)$ if and only if the inequality system \eqref{eq:wopt:np:hard:shortest:inequalities} is feasible for at least one $(\An, \cn) \in (\bAn, \bcn)$, which actually is exactly the problem of testing weak feasibility of the interval linear system with data $(0,0,(\bAn)\T, 0,0, \bcn)$. For such an interval linear system, testing weak feasibility is NP-hard due to Lemma \ref{the:weak:feasibility:np:hard}.

We have that an algorithm for the problem \ref{enu:wopt} can be used to solve an NP-hard problem, hence the problem \ref{enu:wopt} is at least as hard.
\end{proof}

\section{Algorithm for \ref{enu:cwopt}}\label{sec:polynomial:case}

In this section, we describe an algorithm for solving the problem \ref{enu:cwopt}. First, we demonstrate the idea on a simpler case with no equality constraint. Then we show that it can be rewritten to treat interval linear programs in their full generality.

Recall the notation introduced in Definition \ref{def:interval:linear:programming:problem}: the symbol $k$ denotes the number of equality constraints, the number of inequality constraints is denoted by $\ell$.
Our method will be able to test weak optimality of a given point by solving $2^k$ feasibility problems of classical linear systems.

\subsection{The simple case: inequality constraints}
\label{sub:inequalities}

Weak optimality of a given $(\xf,\xn)$ can be tested via solving the nonlinear system \eqref{eq:strong:duality} by Lemma \ref{lem:characterization:weak:optimality}.
Our key idea is the following: if $k=0$, the nonlinear system \eqref{eq:strong:duality} can be rewritten as a linear system.
A nice geometric trick takes place here: a~special form of \emph{disjunctive programming} (see e.g. \cite{balas:1998:DisjunctiveprogrammingProperties}) can be utilized. 
Disjunctive programming is a tool for modelling (a hull of) disjunction of some suitably represented sets (e.g. convex polyhedra). 

It is based on the observation that polyhedra can be scaled by scaling right hand sides only. Consider $P^1 = \{x: A^1 x \le b^1 \}$ and $P^2 = \{x: A^2 x \le b^2\}$. The natural way to express the conic hull of $P^1 \cup P^2$ is 
$$\{x: x = y^1 x^1 + y^2x^2,\ A^1x^1 \le b^1,\ A^2x^2 \le b^2,\ 0\le y^1,\ 0\le y^2\}, $$
which is apparently nonlinear. However, the coefficients $y^1$ and $y^2$ can be moved into the description of $P^1$ and $P^2$, removing nonlinearity:
$$\{z: z = z^1 + z^2,\ A^1z^1 \le y^1 b^1,\ A^2z^2 \le y^2 b^2,\ 0\le y^1,\ 0\le y^2\}. $$
Note that we actually use substitution $z^i = y^i x^i$. This is possible since since $y^i \ge 0$.
Note also that if one of the coefficients, say $y^i$, is zero, it means that $z^i$ is also zero. 

The above linearization applied on the system \eqref{eq:strong:duality} allows for deriving Theorem~
\ref{the:inequality:ilp:polynomial}. Similarly as in the above simple example, we scale the limits of the intervals in $\bBf, \bBn$ and $\bb$ by the corresponding dual variables and substitute new variables $(\Bfp, \Bnp, \bp) = \diag(y^n)(\Bf, \Bn, b)$.

\begin{theorem}
    \label{the:inequality:ilp:polynomial}
    Let an interval linear program $\ILP(\Dp)$ with data 
    $$\Dp = (0,0,\bBf, \bBn, 0, \bb,\bcf,\bcn) $$ 
    be given. 
    
    A given $x = (\xf, \xn)$ is a weakly optimal solution of $\ILP(\Dp)$ if and only if it is a weakly feasible solution of $\ILS((0,0,\bBf, \bBn, 0, \bb))$ and there exists a solution $(\Bfp, \Bnp, \yn)$ of the system
    \begin{subequations}
        \label{eq:inequalities:testing:system}
        \begin{gather}
            \label{eq:inequalities:testing:system:scaling:primal}
            \Bfp \in \diag(\yn)\bBf, \quad \Bnp  \in \diag(\yn)\bBn,\quad \bp\in \diag(\yn)\bb,\\
            \label{eq:inequalities:testing:system:scaled:primal:feasibility}
            \Bfp \xf + \Bnp \xn = \bp,\\
            \label{eq:inequalities:testing:system:scaled:dual:feasibility}
            (e\T \Bfp)\T \in \bcf,\\
            \label{eq:inequalities:testing:system:scaled:dual:feasibility:nonnegative:1}
            (e\T\Bnp)_i \in (\bcn)_i\quad \forall i \in \{\iota| x_\iota > 0\},\\
            \label{eq:inequalities:testing:system:scaled:dual:feasibility:nonnegative:2}
            (e\T\Bnp)_i  \le (\ucn)_i\quad \forall i \in \{\iota| x_\iota = 0\}, \\
            \yn \ge 0.
            \label{eq:inequalities:testing:system:nonnegativity}
        %\diag(\yn) \dBn \le \Bnp \le \diag(\yn)\uBn \\
        %\diag(\yn) \dBn \le \Bnp \le \diag(\yn)\uBn \\
    \end{gather}
    \end{subequations}
    If so, a scenario witnessing weak optimality of $x$ is 
    \begin{equation}
        \label{eq:inequality:witness}s^w = (0,0,\Bf, \Bn, 0, b, (e\T\Bfp)\T, \cn),
    \end{equation} 
    where $i$th row of $\Bf, \Bn, b$ is determined as follows:
    \begin{itemize}
        \item if $\yn_i > 0$, then $(\Bf_i, \Bn_i, b_i) = \frac{1}{\yn_i} (\Bfp_i, \Bnp_i, \bp_i)$, 
        \item else $(\Bf_i, \Bn_i, b_i)$ is determined as a solution of linear system 
            \begin{equation}
            \Bf_i \xf + \Bn_i \xn \ge b_i, \quad \Bf_i \in \bBf_i,\quad \Bn_i \in \bBn_i, \quad b_i \in \bb_i,
                \label{eq:scenario:for:satisfying:i:th:inequality}
            \end{equation}
    \end{itemize}
    and $i$th row of $\cn$ is determined as
    \begin{equation*}
        \cn_i = \left\{
            \begin{array}{ll}
                \ucn_i& \quad\text{if $x_i=0$},\\
                (e\T\Bnp)_i&\quad\text{otherwise}.
    \end{array}
    \right.
    \end{equation*}

\end{theorem}
\begin{proof}
    \hfill

    ``$\Rightarrow$'': We know that $x$ is a weakly optimal solution, hence there is a scenario $\scp = (0,0,\Bf, \Bn, 0, b, \cf, \cn)$ and a vector $\yn$ such that it together solves the system \eqref{eq:strong:duality} for the fixed $x$. Then also $(\Bfp, \Bnp, \bp) = (\diag(\yn) \Bf, \diag(\yn) \Bn, \diag(\yn) b)$ solves the system \eqref{eq:inequalities:testing:system}, since 
    \begin{itemize}[topsep=0pt]
        \item 
     the rows in relations \eqref{eq:inequalities:testing:system:scaling:primal} are only scaled or nullified rows of some constraints in the system \eqref{eq:strong:duality}   
 \item  constrains \eqref{eq:inequalities:testing:system:scaled:dual:feasibility}--\eqref{eq:inequalities:testing:system:nonnegativity} are actually contained in system \eqref{eq:strong:duality}, and
 \item an $i$th row of \eqref{eq:inequalities:testing:system:scaled:primal:feasibility} either has the form $0 = 0$ (if $\yn_i = 0$), or (otherwise) is satisfied via $i$th row of \eqref{eq:strong:duality:primal}, which is satisfied as equality due to $i$th complementarity constraint in \eqref{eq:strong:duality:complementarity}.
    \end{itemize}
    
``$\Leftarrow$'' and ``If so'':
Assume $(\Bfp, \Bnp, \bp, \yn)$ solves \eqref{eq:inequalities:testing:system} for a given weakly feasible $x$. Since $x$ is weakly feasible, a solution $(\Bf_i, \Bn_i, b_i)$ of the system \eqref{eq:scenario:for:satisfying:i:th:inequality} exists for every $i=1,\dots,\ell$. Hence, we can construct the scenario $s^w$ in \eqref{eq:inequality:witness}. 

Now, note that $\yn$ and $s^w$ solve \eqref{eq:strong:duality} (and hence $x$ is weakly optimal):
\begin{itemize}[topsep=0pt]
    \item an $i$th row of \eqref{eq:strong:duality:primal} follows either from rescaling the corresponding row of \eqref{eq:inequalities:testing:system:scaled:primal:feasibility} (for $\yn_i >0$), or from \eqref{eq:scenario:for:satisfying:i:th:inequality} (for $\yn_i = 0$),
    \item the dual feasibility constraint \eqref{eq:strong:duality:dual} is obtained simply by substitution to \eqref{eq:inequalities:testing:system:scaled:dual:feasibility},
    \item the complementarity constraints are obviously satisfied: if $y_i >0$, then $i$th row of primal feasibility is satisfied as equality, if $x_i>0$, $i$th row of dual feasibility is satisfied as equality,
    \item the ``scenario feasibility'' constraints \eqref{eq:characterization:weak:optimality:constraints} are satisfied, since the scenario $s^w$ is is clearly correctly built.\qedhere
\end{itemize}
\end{proof}

\begin{corollary}
    \label{cor:inequalities:polynomially:solvable}
    The problem \ref{enu:cwopt} is polynomially solvable via checking feasibility of a linear system if the underlying interval linear program has no equality constraints (i.e. $k=0$).

    The weak optimality test itself can be done by solving the system \eqref{eq:inequalities:testing:system}. If a~scenario witnessing optimality is also necessary, it can be obtained using \eqref{eq:inequality:witness} by solving additional systems of form \eqref{eq:scenario:for:satisfying:i:th:inequality}.
\end{corollary}

    \newcommand{\xiii}{1}
    \newcommand{\xii}{2}
    \newcommand{\bii}{-1}
    \newcommand{\biu}{3.75}
    \newcommand{\bid}{2}
    \newcommand{\uBiici}{4.5}
    \newcommand{\dBiici}{0}
    \newcommand{\dBiicii}{-1.25}
    \newcommand{\uBiicii}{-0.75}
    \newcommand{\uBivci}{-1.25}
    \newcommand{\dBivci}{-2}
    \newcommand{\uBivcii}{1.5}
    \newcommand{\dBivcii}{0.75}
    \newcommand{\uBiiici}{-2}
    \newcommand{\dBiiici}{-4}
    \newcommand{\uBiiicii}{0}
    \newcommand{\dBiiicii}{-1.5}
    \newcommand{\uBici}{3.5}
    \newcommand{\dBici}{1.5}
    \newcommand{\uBicii}{1.5}
    \newcommand{\dBicii}{0.5}
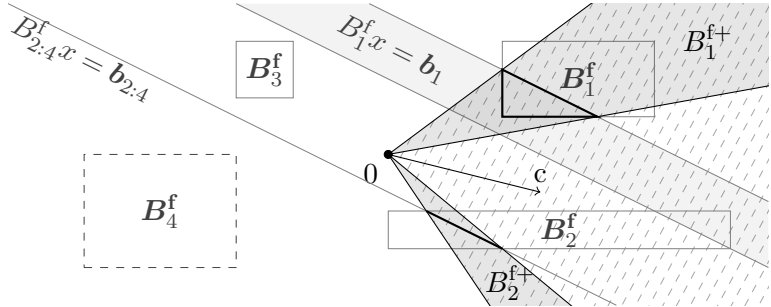
\begin{figure}[b]
    \centering
\begin{tikzpicture}
    \coordinate (lbc) at (-5,-2);
    \coordinate (rtc) at (5,2);
    \pgfdeclarelayer{background}
    \pgfsetlayers{background,main}
    \clip [name path=clip] (lbc) rectangle (rtc);
    %\draw[help lines] (lbc) grid (rtc);
    \coordinate (o)  node [below left] {0};
    \filldraw (o) circle (1.5pt);
    \begin{scope}[help lines]
        \seg[blower1]{\xiii}{\xii}{\bid}
        \seg[bupper1]{\xiii}{\xii}{\biu}
        \seg[b2]{\xiii}{\xii}{\bii}
    \end{scope}

   \fill [opacity=0.05] (blower1-st) -- (blower1-en) -- (bupper1-en) -- (bupper1-st) -- cycle;
   \path[gray!50!black] ($(blower1-st)!0.5!(bupper1-st)$) -- ($(blower1-en)!0.5!(bupper1-en)$) node[pos=0.06,anchor=west,sloped] {$\Bf_{1} x = \bb_1$}; 
   \path[gray!50!black] (b2-st) -- (b2-en) node[pos=0,anchor=west,sloped,yshift=-0.2cm,xshift=0.05cm] {$\Bf_{2:4}x=\bb_{2:4}$}; 

    \draw [name path=B1,help lines] (\dBici,\dBicii) rectangle (\uBici,\uBicii) node [gray!50!black,pos=.5,opacity=1] {$\bBf_{1}$} ;
    \draw [name path=B2,help lines] (\dBiici,\dBiicii) rectangle (\uBiici,\uBiicii) node [gray!50!black,pos=.5,opacity=1] {$\bBf_{2}$} ;
    \draw [name path=B3,help lines,dashed,gray!50!black] (\dBiiici,\dBiiicii) rectangle (\uBiiici,\uBiiicii) node [gray!50!black,pos=.5,opacity=1] {$\bBf_{4}$} ;
    \draw [name path=B4,help lines] (\dBivci,\dBivcii) rectangle (\uBivci,\uBivcii) node [gray!50!black,pos=.5,opacity=1] {$\bBf_{3}$} ;

    \path [name intersections={of=B1 and bupper1,name=cone1}];
    \path [name intersections={of=B2 and b2,name=cone2}];

    \path (0,0) -- (cone1-1) --([turn]0:100cm) coordinate (cone1-limit-1);
    \path (0,0) -- (cone1-2) --([turn]0:100cm) coordinate (cone1-limit-2);
    \path (0,0) -- (cone2-1) --([turn]0:100cm) coordinate (cone2-limit-1);
    \path (0,0) -- (cone2-2) --([turn]0:100cm) coordinate (cone2-limit-2);
    
    \filldraw [fill opacity=0.1,name path=cone1] (0,0)--(cone1-limit-1)--(cone1-limit-2)--cycle;
    \filldraw [fill opacity=0.1,name path=cone2] (0,0)--(cone2-limit-1)--(cone2-limit-2)--cycle;
    \path [name intersections={of=cone2 and b2,name=Bset2}];
    \path [name intersections={of=cone1 and bupper1,name=Bset1}];
    \path (current bounding box.north east) -- ++(-0.8,-0.5) node {$\Bfp_1$};
    \path (Bset2-2) -- ++(0.1,-0.45) node {$\Bfp_2$};

    \draw[thick] (Bset2-1) -- (Bset2-2);
    \draw[thick] (Bset1-1) -- (Bset1-2) -- (\dBici,\dBicii) --cycle;

    \draw[->] (0,0)--(2,-0.5) node[above] {c};

    \begin{pgfonlayer}{background}
    \clip  (lbc) rectangle (rtc);
    \begin{scope}[loosely dashed,very thin,gray!70!white]
    \foreach \factor in {0.2,0.4,...,7} {
        \draw[dashed] ($\factor*(Bset1-1)$) -- ($\factor*(Bset2-1)$);
    }
\end{scope}\end{pgfonlayer}
    
\end{tikzpicture}
\caption{Illustration of Example \ref{exa:illustration}. The space of rows of the matrix $\Bf$ is depicted. For example, the line $\Bf_{2:4}x=\bb_{2:4}$ contains pairs of coefficients $(\Bf_{2,1},\Bf_{2,2}), (\Bf_{3,1}, \Bf_{3,2}), (\Bf_{4,1},\Bf_{4,2})$ such that second to fourth constraint is satisfied as equality for the given $x$.}
    \label{fig:illustration}
\end{figure}

\begin{example}
    \label{exa:illustration}
    We demonstrate the idea on a small example.
    The geometry behind Theorem~\ref{the:inequality:ilp:polynomial} is depicted on Figure \ref{fig:illustration}. The figure shows the space of coefficients of constraints.
    
 The setting is the following: assume that we are given interval linear program 
 \begin{equation*}
  \begin{alignedat}{2}%
    \min\ &&\phantom{=}
    (\cf)\T  \xf  &\quad \text{s.t.}   \\
&&\phantom{=}\bBf \xf  &\ge \bb
\end{alignedat}
 \end{equation*}
  with
\begin{equation}
    \bBf = \left(
    \begin{array}{rlcrl}
        [\dBici,&\uBici]&, & [\dBicii,& \uBicii]  \\{}%
        [\dBiici,&\uBiici]&, & [\dBiicii,& \uBiicii]  \\{}%
        [\dBivci,&\uBivci]&, & [\dBivcii,& \uBivcii]%
    \end{array}
    \right),\quad
    \bb = \begin{pmatrix}
        [\bid,\biu] \\
        \bii\\
        \bii
    \end{pmatrix},\quad
    \cf = \left(\begin{matrix}
    2 \\
    -0.5\end{matrix}\right),
\end{equation}
i.e. there are two free variables and the  objective function is crisp. 
We are to test weak optimality for $x = \left(\begin{smallmatrix}
    \xii \\
    \xiii
\end{smallmatrix}\right)$.

Note that the third constraint can't be satisfied as equality for the given $x$. This enforces $y_3 = 0$. Hence the third rows of \eqref{eq:inequalities:testing:system:scaling:primal} and \eqref{eq:inequalities:testing:system:scaled:primal:feasibility} are null. 

Our given $x$ is weakly optimal if $c$ is in the conic hull of all the feasible $\Bf_1 \in \bBf_1$ and $\Bf_2\in \bBf_2$ (see the bold triangle and the bold line segment in the figure). 
 The conic hull itself   is depicted with dashed pattern. Note that it corresponds to left hand sides of the constraints \eqref{eq:strong:duality:dual} and \eqref{eq:inequalities:testing:system:scaled:dual:feasibility}.

 The gray cones are cones of all the feasible $\Bfp_1$ and $\Bfp_2$.
Since their conic hull contains $c$, we have that $x$ is weakly optimal.

Just for illustration, there is also an additional fourth constraint $$([\dBiiici,\uBiiici], [\dBiiicii, \uBiiicii]) \xf \ge \bii$$ in the figure. There is no $\Bf_4$ such that $x$ is feasible. With this fourth constraint $x$ is not weakly optimal.
\end{example}

\subsection{The general case: equations are allowed}
\label{sub:the_general_case}

Now consider the problem \ref{enu:cwopt} in its full generality. The key in the special case was in the nonnegativity of $\yn$. Assuming equality constraints, we need to consider also free dual variables $\yf$. However, the linearization based on disjunctive programming requires variables with a fixed sign (nonnegative or nonpositive ones).
The underlying problem is that there is no way to express a nonconvex set by means of linear systems, as shows Example \ref{exa:illustration2}:

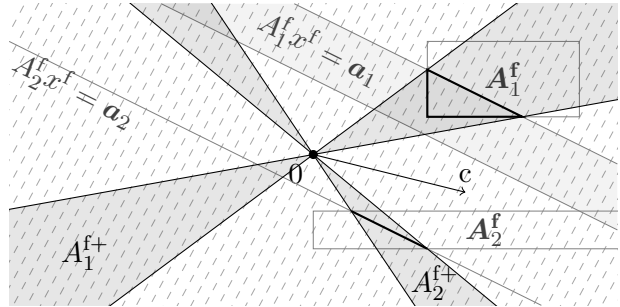
\begin{figure}[b]
    \centering
\begin{tikzpicture}
    \coordinate (lbc) at (-4,-2);
    \coordinate (rtc) at (4,2);
    \pgfdeclarelayer{background}
    \pgfsetlayers{background,main}
    \clip [name path=clip] (lbc) rectangle (rtc);
    %\draw[help lines] (lbc) grid (rtc);
    \coordinate (o)  node [below left] {0};
    \filldraw (o) circle (1.5pt);
    \begin{scope}[help lines]
        \seg[blower1]{\xiii}{\xii}{\bid}
        \seg[bupper1]{\xiii}{\xii}{\biu}
        \seg[b2]{\xiii}{\xii}{\bii}
    \end{scope}

   \fill [opacity=0.05] (blower1-st) -- (blower1-en) -- (bupper1-en) -- (bupper1-st) -- cycle;
   \path[gray!50!black] ($(blower1-st)!0.5!(bupper1-st)$) -- ($(blower1-en)!0.5!(bupper1-en)$) node[pos=0.06,anchor=west,sloped] {$\Af_{1} \xf = \ba_1$}; 
   \path[gray!50!black] (b2-st) -- (b2-en) node[pos=0,anchor=west,sloped,yshift=-0.2cm,xshift=0.05cm] {$\Af_{2}\xf=\ba_{2}$}; 

    \draw [name path=B1,help lines] (\dBici,\dBicii) rectangle (\uBici,\uBicii) node [gray!50!black,pos=.5,opacity=1] {$\bAf_{1}$} ;
    \draw [name path=B2,help lines] (\dBiici,\dBiicii) rectangle (\uBiici,\uBiicii) node [gray!50!black,pos=.5,opacity=1] {$\bAf_{2}$} ;
     %[name path=B3,help lines,dashed] (\dBiiici,\dBiiicii) rectangle (\uBiiici,\uBiiicii) node [gray,pos=.5,opacity=1] {$\bB_{4}$} ;
    %\draw [name path=B4,help lines] (\dBivci,\dBivcii) rectangle (\uBivci,\uBivcii) node [gray,pos=.5,opacity=1] {$\bB_{3}$} ;

    \path [name intersections={of=B1 and bupper1,name=cone1}];
    \path [name intersections={of=B2 and b2,name=cone2}];

    \path (0,0) -- (cone1-1) --([turn]0:100cm) coordinate (cone1-limit-1);
    \path (0,0) -- (cone1-2) --([turn]0:100cm) coordinate (cone1-limit-2);
    \path (0,0) -- (cone1-1) --([turn]180:100cm) coordinate (cone1-limit-3);
    \path (0,0) -- (cone1-2) --([turn]180:100cm) coordinate (cone1-limit-4);
    \path (0,0) -- (cone2-1) --([turn]0:100cm) coordinate (cone2-limit-1);
    \path (0,0) -- (cone2-2) --([turn]0:100cm) coordinate (cone2-limit-2);
    \path (0,0) -- (cone2-1) --([turn]180:100cm) coordinate (cone2-limit-3);
    \path (0,0) -- (cone2-2) --([turn]180:100cm) coordinate (cone2-limit-4);
    
    \filldraw [fill opacity=0.1,name path=cone1] (0,0)--(cone1-limit-1)--(cone1-limit-2)--cycle;
    \filldraw [fill opacity=0.1,name path=cone2] (0,0)--(cone2-limit-1)--(cone2-limit-2)--cycle;
    \filldraw [fill opacity=0.1,name path=cone3] (0,0)--(cone1-limit-3)--(cone1-limit-4)--cycle;
    \filldraw [fill opacity=0.1,name path=cone4] (0,0)--(cone2-limit-3)--(cone2-limit-4)--cycle;
    \path [name intersections={of=cone2 and b2,name=Bset2}];
    \path [name intersections={of=cone1 and bupper1,name=Bset1}];
    \path (current bounding box.south west) -- ++(1,0.7) node {$\Afp_1$};
    \path (Bset2-2) -- ++(0.1,-0.45) node {$\Afp_2$};

    \draw[thick] (Bset2-1) -- (Bset2-2);
    \draw[thick] (Bset1-1) -- (Bset1-2) -- (\dBici,\dBicii) --cycle;

    \draw[->] (0,0)--(2,-0.5) node[above] {c};

    \begin{pgfonlayer}{background}
    \clip  (lbc) rectangle (rtc);
    \begin{scope}[loosely dashed,very thin,gray!70!white]
    \foreach \factor in {-7,-6.8,...,-0.2,0.001,0.2,0.4,...,7} {
        \draw[dashed,shorten >=-3cm,shorten <=-3cm] ($\factor*(Bset1-1)$) -- ($\factor*(Bset2-1)$);
    }
\end{scope}\end{pgfonlayer}
    
\end{tikzpicture}
\caption{Illustration of Example \ref{exa:illustration2}. The space of rows of the matrix $\Af$ is depicted. For example, the line $\Af_{2}x=\bb_{2}$ contains pairs of coefficients $(\Af_{2,1},\Af_{2,2})$ such that second constraint is satisfied.}
    \label{fig:illustration:2}
\end{figure}

\begin{example}
    \label{exa:illustration2}
    Consider now the ILP from Example \ref{exa:illustration} with equalities instead of inequalities (and also only with the first two constraints):
 \begin{equation*}
  \begin{alignedat}{2}%
    \min\ &&\phantom{=}
    (\cf)\T  \xf  &\quad \text{s.t.}   \\
&&\phantom{=}\bAf \xf  &= \ba,
\end{alignedat}
 \end{equation*}
  where
\begin{equation}
    \bAf = \left(
    \begin{array}{rlcrl}
        [\dBici,&\uBici]&, & [\dBicii,& \uBicii]  \\{}%
        [\dBiici,&\uBiici]&, & [\dBiicii,& \uBiicii] 
    \end{array}
    \right),\quad
    \ba = \begin{pmatrix}
        [\bid,\biu] \\
        \bii\\
    \end{pmatrix},\quad
    \cf = \left(\begin{matrix}
    2 \\
    -0.5\end{matrix}\right).
\end{equation}
Again, we are to test weak optimality for $x = \left(\begin{smallmatrix}
    \xii \\
    \xiii
\end{smallmatrix}\right)$.

The space of $\Af_1$ and $\Af_2$ is depicted on Figure \ref{fig:illustration:2}. The meaning of elements of the figure is analogous to Figure \ref{fig:illustration}. 

The sets of all primarily feasible normals $\Af_1$ and $\Af_2$ can now be scaled by both positive and negative factors. The resulting sets are nonconvex double-cones. 

Note that for a classical linear program, these double-cones are degenerated to a~single line, which is actually a~convex set.
\end{example}

To emphasize: from the perspective of the method presented in Section \ref{sub:inequalities}, the problem is that we cannot linearize the dual feasibility constraint, because we need to sum up points from possibly nonconvex sets. Alternatively, the problem can be viewed in the constraints \eqref{eq:inequalities:testing:system:scaling:primal}. Consider a constraint $\be \in y \bbe$ for some $\be \in \R$ and $\bbe \in \IR$. If $y \ge 0$ (or $y \le 0$), it can rewritten as $y \dbe \le \be \le y\ube$ (or $y \ube \le \be \le y\dbe$), however, if $y$ can be arbitrary, the two cases must be distinguished.

However, the standard orthant decomposition of $\yf$-space can be apparently used here. We do so in Theorem \ref{the:algorithm:for:general:case}.

\begin{definition}
    \label{def:orthant:problem}
    For given data $\Dp = (\bAf, \bAn, \bBf, \bBn, \ba, \bb, \bcf, \bcn)$ of an ILP and a given sign vector $\sigma \in \{-1,1\}^{k}$, the \emph{testing system for $\Dp$ in orthant $\sigma$} is the system  \eqref{eq:general:system} in the form
    \begin{subequations}
        \label{eq:general:system}
        \begin{gather}
            \label{eq:general:system:scaling:coeffs:equalities}
            \Afp \in \diag(\yf)\bAf, \quad \Anp  \in \diag(\yf)\bAn,\quad \ap\in \diag(\yf)\ba,\\
            \label{eq:general:system:scaling:coeffs:inequalities}
            \Bfp \in \diag(\yn)\bBf, \quad \Bnp  \in \diag(\yn)\bBn,\quad \bp\in \diag(\yn)\bb,\\
            \label{eq:general:system:primal:feasibility}
            \Afp \xf + \Anp \xn = \ap, \qquad \Bfp \xf + \Bnp \xn = \bp\\
            \label{eq:general:system:dual:feasibility}
            (e\T \Afp + e\T \Bfp)\T \in \bcf\\
            \label{eq:general:system:scaled:dual:feasibility:nonnegative:1}
            (e\T\Anp+e\T\Bnp)_i \in (\bcn)_i\quad \forall i \in \{\iota| x_\iota > 0\},\\
            \label{eq:general:system:scaled:dual:feasibility:nonnegative:2}
            (e\T\Anp+e\T\Bnp)_i  \le (\ucn)_i\quad \forall i \in \{\iota| x_\iota = 0\}, \\
            \yn \ge 0,
            \label{eq:general:system:nonnegativity}\\
            \diag(\sigma)\yf \ge 0.
            \label{eq:general:system:orthant}
        %\diag(\yn) \dBn \le \Bnp \le \diag(\yn)\uBn \\
        %\diag(\yn) \dBn \le \Bnp \le \diag(\yn)\uBn \\
    \end{gather}
    \end{subequations}
\end{definition}
\begin{theorem}
    \label{the:algorithm:for:general:case}
    Assume ILP with data $\Dp=(\bAf, \bAn, \bBf, \bBn, \ba, \bb, \bcf, \bcn)$. A solution $x = (\xf,\xn) $ is weakly optimal if and only if $x$ is a weakly feasible solution of $\ILP(\Dp)$ and there is $\sigma \in \{-1,1\}^{k}$ such that the testing system for $\Dp$ in orthant $\sigma$ is feasible for the fixed $x$. If so, a scenario witnessing optimality of $x$ is 
    \begin{equation}
        \label{eq:general:witness}s^w = (\Af,\An,\Bf, \Bn, a, b, (e\T\Afp\!\!+\!e\T\Bfp)\T, \cn),
    \end{equation} 
    where $i$th row of $\Af, \An, a$ is determined as follows:
    \begin{itemize}
        \item if $\yf_i\not = 0$, then $(\Af_i, \An_i, a_i) = \frac{1}{\yf_i} (\Afp_i, \Anp_i, \ap_i)$, 
        \item else $(\Af_i, \An_i, a_i)$ is determined as a solution of the linear system
    \end{itemize}
            \begin{equation}
            \Af_i \xf + \An_i \xn = a_i, \quad \Af_i \in \bAf_i,\quad \An_i \in \bAn_i, \quad a_i \in \ba_i,
                \label{eq:general:scenario:B}
            \end{equation}
    $i$th row of $\Bf, \Bn, b$ is determined as follows:
    \begin{itemize}
        \item if $\yn_i > 0$, then $(\Bf_i, \Bn_i, b_i) = \frac{1}{\yn_i} (\Bfp_i, \Bnp_i, \bp_i)$, 
        \item else $(\Bf_i, \Bn_i, b_i)$ is determined as a solution of \eqref{eq:scenario:for:satisfying:i:th:inequality},
    \end{itemize}
    and $i$th row of $\cn$ is determined as
    \begin{equation*}
        \cn_i = \left\{
            \begin{array}{ll}
                \ucn_i& \quad\text{if $x_i=0$},\\
                (e\T\Anp+e\T\Bnp)_i&\quad\text{otherwise}.
    \end{array}
    \right.
    \end{equation*}
   
\end{theorem}
\begin{proof}
    Very analogous to the proof of Theorem \ref{the:inequality:ilp:polynomial}.
\end{proof}

\begin{corollary}
    The problem \ref{enu:cwopt} for an ILP with data $\Dp$ can be solved by solving $2^k$ linear systems, one for every $\sigma \in \{-1,1\}^k$. Recall that $k$ is the number of equality constraints in the ILP. The size of the linear systems to be solved is linear in the number of variables and constraints of the ILP.
\end{corollary}

\section{Conclusions}
\label{sec:conclusion}

We proved that the problem of testing weak optimality of a given solution of a~given interval linear program is NP-hard. We proposed an algorithm, based on orthant decomposition, which can decide the problem via solving of $2^k$ linear systems, where $k$ is the number of equality constraints in the given interval linear program.
In particular, this means that the proposed method works in polynomial time for interval linear programs with inequality constraints only.

\bibliographystyle{spmpsci}      % mathematics and physical sciences
\bibliography{ref}   % name your BibTeX data base

%\begin{thebibliography}{}
%%
%% and use \bibitem to create references. Consult the Instructions
%% for authors for reference list style.
%\end{thebibliography}

\end{document}